\documentclass{birkjour}
\usepackage{amsmath,amssymb,amsthm,mathtools,mathrsfs}

\usepackage{bm}

 \newtheorem{thm}{Theorem}[section]
 
 \newtheorem{lem}[thm]{Lemma}
 \newtheorem{prop}[thm]{Proposition}
 \theoremstyle{definition}
 
 \theoremstyle{remark}
 \newtheorem{rem}[thm]{Remark}
 
 \numberwithin{equation}{section}

\newcommand{\R}{\mathbb{R}}

\newcommand{\N}{\mathbb{N}}
\newcommand{\PP}{\mathbb{P}}
\newcommand{\EE}{\mathbb{E}}

\DeclareMathOperator{\dive}{div}
\DeclareMathOperator{\curl}{curl}
\newcommand{\norm}[1]{\left\lVert #1 \right\rVert}

\newcommand{\Ezero}[1]{\mathcal{E}_0^{#1}}
\newcommand{\Eone}[1]{\mathcal{E}_1^{#1}}

\usepackage[hidelinks,hypertexnames=false]{hyperref}

\begin{document}

%-------------------------------------------------------------------------
% editorial commands: to be inserted by the editorial office
%
%\firstpage{1} \volume{228} \Copyrightyear{2004} \DOI{003-0001}
%\seriesextraline{This is the Concrete Title of this Book\br H.E. R and S.T.C. W, Eds.}
%\firstpage{1}

\title[Uniqueness of invariant measures for stochastic anisotropic NS]{Uniqueness of invariant measures for  stochastic damped anisotropic Navier--Stokes equations}

%----------Author 1
\author[]{Siyu Liang}

\address{School of Mathematics and Statistics, Nanjing University of Science and Technology, Nanjing, China;\\
 Department of Mathematics, Bielefeld University, D-33615 Bielefeld, Germany
}

\email{liangsiyu1994@163.com}

\thanks{}
%----------Author 2

\subjclass{Primary 60H15; Secondary 35Q35, 37A25 }

\keywords{anisotropic Navier--Stokes; damping; Hilbert--Schmidt norm; coupling method; invariant measures.}

\date{\today}
\dedicatory{}
%%% ----------------------------------------------------------------------

\begin{abstract}
We study a two-dimensional Navier--Stokes system with anisotropic viscosity, linear damping term, and an additive noise on the whole space $\mathbb{R}^2$. For this model we prove uniqueness of invariant measures when the damping coefficient is sufficiently large compared to the noise intensity. The argument is based on an asymptotic coupling method and relies on anisotropic energy estimates together with exponential-type estimates for the $H^1$-energy. Since no Poincar\'e inequality is available on $\mathbb{R}^2$, the damping term is essential even for the existence of invariant measures. Our result applies to general additive noise without any non-degeneracy condition and remains valid even in the deterministic case $\sigma\equiv0$.
\end{abstract}

%%% ----------------------------------------------------------------------
\maketitle
%%% ----------------------------------------------------------------------
%\tableofcontents
\section{Introduction}
    In this paper, we consider the following damped anisotropic Navier--Stokes equation on  $\R^2$  \begin{equation}
\label{eq:main}
\left\{
\begin{aligned}
&\partial_t u + u\cdot\nabla u - \partial_1^2 u + \lambda u = -\nabla p + \sigma\,\mathrm{d}W,\\
&\dive u = 0,\\
&u\big|_{t=0} = u_0,
\end{aligned}
\right.
\end{equation}
 where $\lambda>0$ is the damping coefficient, $W$ is an $\ell^2$-cylindrical Wiener process, and $\sigma$ is an additive noise operator.
Our aim is to prove that the Markov semigroup associated with \eqref{eq:main} admits a unique invariant measure when the damping is sufficiently strong compared to the noise intensity.

 The anisotropic structure in \eqref{eq:main} arises naturally in geophysical fluid models, where the horizontal and vertical viscosities are different. 
 The  structure  has been investigated in various deterministic and stochastic settings. In our previous work \cite{LiangZhangZhu2021}, the undamped equation (i.e. $\lambda=0$) was studied under anisotropic $L^2$-regularity assumptions on $u_0$ and $\partial_2 u_0$. The present paper focuses instead on the long-time behaviour of the damped dynamics. We show that when $\lambda$ dominates the Hilbert--Schmidt norm of the noise, the system has at most one invariant measure.

We remark that the damping term $\lambda u$ is essential in our setting.
Unlike the bounded-domain results of Flandoli and Gatarek~\cite{FlandoliGatarek1995}, where a Poincar\'e inequality provides a natural control of the $L^2$-energy, no such inequality holds on the whole space $\mathbb{R}^2$. Therefore, even the existence of invariant measures requires an additional large-scale dissipation, which is provided here by the linear damping. This  has already appeared in earlier papers on stochastic Navier--Stokes equations and other models on unbounded domains, see, for instance, \cite{BessaihFerrario2014,BrzezniakFerrario2019,BarbuDaPrato2002,BessaihFerrario2020,BrzezniakFerrarioZanella2023}.

Another related work is the paper by Sun, Qiu and Tang~\cite{SunQiuTang2024}, which established ergodicity for anisotropic stochastic Navier--Stokes and primitive equations on bounded domains by means of an asymptotic coupling argument. In their setting on bounded domains, a Poincar\'e inequality provides
sufficient dissipation so that no damping term is required. Moreover, their analysis relies on a nontrivial stochastic forcing that plays a crucial role in the coupling construction. In contrast, the whole-space case considered here lacks any Poincar\'e inequality, hence we  introduce a linear damping term to recover the necessary energy control at large scales. Note that our argument does not rely on any non-degeneracy of the noise and remains valid even in the deterministic case $\sigma \equiv 0$.

The proof combines anisotropic energy estimates, exponential-type estimates for the $H^{1}$-energy, and an asymptotic coupling method adapted to our setting.  The linear damping term provides the missing large-scale coercivity since there is no
Poincar\'e inequality on $\R^{2}$.  The most important step is the use of anisotropic Gagliardo--Nirenberg interpolation inequalities, which take advantage of the mixed horizontal--vertical structure of the dissipation. 
 The anisotropic estimates are applied to control the nonlinear term in terms of the anisotropic dissipation. Exponential martingale inequalities for the $L^{2}$ and $H^{1}$ cases then yield a set of positive probability on which the solution satisfies bounds which are uniform in time. On this set, a Gr\"onwall argument gives an exponential decay estimate for the difference of two solutions and an asymptotic coupling criterion finally implies uniqueness of invariant measures.

The structure of the paper is as follows:
Section~\ref{Sect2} introduces the notations, Section~\ref{sec3} recalls several preliminary results,
and  
Section~\ref{sec4} contains the main theorem and its proof.
The precise formulation of our main result is given in Theorem \ref{thm:unique}.

\section{Preliminaries}\label{Sect2}
We first recall some function spaces on $\mathbb{R}^{2}$.

\subsection{Function spaces on $\mathbb{R}^{2}$}
On $\mathbb{R}^{2}$, we recall the classical Sobolev spaces:
$$
H^{s}(\mathbb{R}^{2}):=\Bigl\{u\in \mathcal{S}'(\mathbb{R}^{2});\ \| u \|_{H^{s}(\mathbb{R}^{2})}^{2}:=\int_{\mathbb{R}^{2}}(1+| \xi |^{2})^s|\hat{u}(\xi)|^{2}\,\mathrm d\xi<\infty\Bigr\},
$$
where $s\geq 0$, and $\hat{u}$ denotes the Fourier transform of $u$.

\subsection{Some other notations}

  Firstly, since throughout this paper we only consider the case of $\R^2$, we omit the domain of the function and vector spaces when no confusion occurs.

  Let $(\Omega,\mathcal{F},(\mathcal{F}_t)_{t\ge0},\PP)$ be a filtered probability space
satisfying the usual conditions and  $W$ be an $\ell^2$-cylindrical Wiener process
with respect to $(\mathcal{F}_t)_{t\ge0}$.
   For a Hilbert space $X$, we use the notation $L_{2}(\ell^2,X)$ to denote the Hilbert--Schmidt operators from $\ell^2$ to $X$ with the associated Hilbert--Schmidt norm. Throughout the paper we use the following notations:
\begin{itemize}
  \item Denote by $u=(u_1,u_2)$  vector fields on  $\mathbb{R}^2$.
  \item Denote by $\nabla=(\partial_1,\partial_2)$  the gradient, and $\Delta=\partial_1^2+\partial_2^2$ the Laplacian.
  \item For a scalar function $f$, define $\nabla^\perp f := (-\partial_2 f, \partial_1 f)$.
  \item For a vector field $u=(u_1,u_2)$, we write
        $$
        \dive u = \partial_1 u_1 + \partial_2 u_2,
        \qquad
        w=\curl u = \partial_1 u_2 - \partial_2 u_1.
        $$
  \item        The $L^2$ inner product and norm are denoted by 
$(f,g) = \int_{\mathbb{R}^2} f g \, \mathrm{d}x.$
  \item We write 
  $H := \{ u \in L^2(\mathbb{R}^2;\mathbb{R}^2): \dive u = 0 \}$
   and $\tilde H^s := \{ u \in H^s(\mathbb{R}^2;\mathbb{R}^2): \dive u = 0 \}$.
 
    \item
        We denote by $x=(x_1,x_2)=(x_{\mathrm h},x_{\mathrm v})\in\R^2$ the horizontal and vertical variables,
and write $\partial_1=\partial_{x_1}$ and $\partial_2=\partial_{x_2}$ for the corresponding derivatives.
For $p,q\in[1,\infty]$, we define the anisotropic mixed norm
$$
\|f\|_{L_{\mathrm h}^p(L_{\mathrm v}^q)}
:= \Big(\int_{\R} \|f(x_{\mathrm h},\cdot)\|_{L_{\mathrm v}^q}^p\,\mathrm{d}x_{\mathrm h}\Big)^{1/p},
$$
and use similar notation for $L_{\mathrm v}^p(L_{\mathrm h}^q)$. (These anisotropic mixed norms are standard in the analysis of the two-dimensional anisotropic Navier--Stokes equations (see, for example, \cite{LiangZhangZhu2021})).

    \end{itemize}
    
\section{Background Results}\label{sec3}
We first introduce the following cancellation for the nonlinear term, which holds only in the two-dimensional case.
\begin{lem}\label{lem:cancel}
Let $u\in\tilde H^2$ and $\omega=\curl u$. Then
\begin{equation}\label{eq:cancel}
\int_{\R^2} (u\cdot\nabla u)\cdot \Delta u\,\mathrm{d}x = 0 .
\end{equation}
\end{lem}

\begin{proof}
Using that $\Delta u=\nabla^\perp\omega$ in two dimensions, we have
$$
\int_{\R^2} (u\cdot\nabla u)\cdot \Delta u\,\mathrm{d}x
= \int_{\R^2} (u\cdot\nabla u)\cdot \nabla^\perp \omega\,\mathrm{d}x
= -\int_{\R^2} \omega\,\curl\!\big((u\cdot\nabla)u\big)\,\mathrm{d}x .
$$
For incompressible planar flows, $\curl((u\cdot\nabla)u)=-(u\cdot\nabla)\omega$, hence
$$
-\int_{\R^2} \omega\,\curl\!\big((u\cdot\nabla)u\big)\,\mathrm{d}x
= \int_{\R^2} \omega\,(u\cdot\nabla)\omega\,\mathrm{d}x
= \tfrac12\int_{\R^2} u\cdot\nabla(\omega^2)\,\mathrm{d}x.
$$
Let $\chi_R\colon \R^2\to \R$ be a smooth cutoff function supported in $B_{2R}$, satisfying that
$\chi_R\equiv1$ on $B_R$ and $|\nabla\chi_R|\lesssim R^{-1}$. Then
\begin{equation}
\begin{split}
\int_{\R^2} u\cdot\nabla(\omega^2)\,\mathrm{d}x
&=\lim_{R\to\infty}\int_{\R^2} \chi_R\,u\cdot\nabla(\omega^2)\,\mathrm{d}x\\
&= -\lim_{R\to\infty}\int_{\R^2} \omega^2\,\mathrm{div}(\chi_R u)\,\mathrm{d}x\\
&= -\lim_{R\to\infty}\int_{\R^2} \omega^2\,u\cdot\nabla\chi_R\,\mathrm{d}x,
\end{split}
\end{equation}
where the last equality follows from $\dive u=0$.
Since $\omega\in L^2(\R^2)$ and $u\in\tilde H^2\subset L^\infty(\R^2)$, we deduce that
$$
\Big|\int_{\R^2} \omega^2\,u\cdot\nabla\chi_R\,\mathrm{d}x\Big|
\lesssim \frac{1}{R}\,\|\omega\|_{L^2}^2\,\|u\|_{L^\infty}
\longrightarrow 0 \quad \text{as } R\to\infty.
$$
This completes the proof.
\end{proof}

The existence of martingale solutions and pathwise uniqueness for \eqref{eq:main} with $u_0\in H^1(\R^2)$ and additive noise $\sigma\in L_2(\ell^2,\tilde H^1)$ can be established by adapting the methods of \cite{LiangZhangZhu2021}. In particular, a unique global probabilistically strong solution exists. The $H^1$-moment estimates follow from It\^o's formula on Galerkin approximations, the cancellation \eqref{eq:cancel}, and the Burkholder--Davis--Gundy and Young inequalities, after which one passes to the limit. 
We collect these standard results as follows:

\begin{thm}[Well-posedness and $H^1$-energy estimates]\label{thm:wp}
Let $u_0\in \tilde H^1$ and $\sigma\in L_2(\ell^2, \tilde H^1)$.
Then \eqref{eq:main} admits a unique global probabilistically strong solution
$$
u(\cdot;u_0)\in C\big([0,\infty); H^1_{\mathrm w}\big)\cap L^2_{\mathrm{loc}}(0,\infty; \tilde H^1),
$$
where $H^1_{\mathrm w}$ denotes $H^1$ equipped with the weak topology.
Pathwise uniqueness implies that for each $u_0\in\tilde H^1$ the solution map
$u(t;u_0)$ defines a Markov process, and the corresponding transition operators
$$
(P_t\phi)(u_0):=\EE[\phi(u(t;u_0))],\qquad \phi\in B_b(\tilde H^1),
$$
form a Markov semigroup on $(\tilde H^1,\mathcal B(\tilde H^1))$.
Moreover, for every $T>0$ there exists $C_T>0$ such that
\begin{equation}\label{solution-H1-estimate}
\EE\!\Big(\sup_{t\in[0,T]}\|u(t)\|_{H^1}^{2}\Big)
\le C_{T}\Big(1+\|u_0\|_{H^1}^{2}+\|\sigma\|_{L_2(\ell^2,H^1)}^{2}\Big)
\end{equation}
and
\begin{equation}\label{solution-H1-dissipation}
\EE\!\Big(\int_0^T \|\partial_1\nabla u(s)\|_{L^2}^2\,\mathrm ds\Big)
\le C_{T}\Big(1+\|u_0\|_{H^1}^{2}+\|\sigma\|_{L_2(\ell^2,H^1)}^{2}\Big).
\end{equation}
\end{thm}

\section{The Maslowski--Seidler Criterion and Existence of Invariant Measures}\label{existence}
For the sake of completeness, before turning to uniqueness, we use the 
Maslowski--Seidler Criterion to show the existence of the invariant measure.
By pathwise uniqueness and existence of probabilistically strong solutions, the solution defines a Markov process on $\tilde H^1$, and the associated transition semigroup $(P_t)_{t\ge0}$ is well defined.
In this section we prove the existence of the 
invariant measure, i.e., we prove the following theorem.
\begin{thm}
{\sl Let $\lambda>0$ and additive noise operator $\sigma\in L_2(\ell^2, \tilde H^1)$. Then
there exists (at least) an invariant measure with respect to the Markovian semigroup $(P_t)_{t\ge0}$.}
\end{thm}
Note that the statement of the existence theorem  does not depend on  the value of $\lambda$ as long as
it is strictly positive. 
\textbf{Hence throughout  Section \ref{existence} we consider  $\lambda$ as a fixed constant and  denote by $C$ a constant which may depend on $\lambda$ and these constants may differ from line to line.}

Follow a similar strategy as in   
\cite{BrzezniakMotylOndrejat2017}, we prove the existence of the invariant measure via Maslowski--Seidler Criterion, 
see, e.g.,  \cite{MaslowskiSeidler1999}.

For all separable Hilbert space $\mathcal{H}$, denote by $\mathcal{H}_w$ the space $\mathcal{H}$  endowed with weak topology. Moreover,  we denote by $\mathscr{C}_b(\mathcal{H}_w)$ the space of bounded continuous functions on $\mathcal{H}_w$ (i.e., bounded weakly continuous functions on $\mathcal{H}$)  and by 
$\mathscr{S}_b(\mathcal{H}_w)$ the space of bounded sequentially weakly continuous functions on $\mathcal{H}$.
Obviously we have $\mathscr{C}_b(\mathcal{H}_w)\subseteq \mathscr{S}_b(\mathcal{H}_w)$.

Let us recall  the following proposition from Proposition 3.1 of \cite{MaslowskiSeidler1999}:
\begin{prop}\label{prop:MS}
Suppose that the semigroup $(P_t)$ is sequentially weakly Feller, that is,
$$
P_t\big(\mathscr{C}_b(\mathcal{H}_w)\big) \subset \mathscr{S}_b(\mathcal{H}_w).
$$
Assume that we can find a Borel probability measure $\nu$ on $\mathcal{H}$ and
$T_0 \ge 0$ such that for any $\varepsilon > 0$ there exists $R > 0$ satisfying
\begin{equation}
\label{eq:MS}
\sup_{T \ge T_0} \frac{1}{T} \int_0^T
P_t^{\ast}\nu\big(\{\|x\|_{\mathcal{H}} > R\}\big)\, dt < \varepsilon .
\end{equation}
Then there exists an invariant probability measure for $(P_t)$.
\end{prop}

We will use the above proposition to prove the existence of the invariant measure.
We divide the proof into two parts, which are exactly the following two subsections.

\subsection{The proof of sequentially weakly Feller}\label{sub41} 
First we will show $P_t\big(\mathscr{C}_b(\tilde H^1_w)\big) \subset \mathscr{S}_b(\tilde H^1_w)$.  
It is equivalent to show that for all $T>0$, it holds that 
$P_t\big(\mathscr{C}_b(\tilde H^1_w)\big) \subset \mathscr{S}_b(\tilde H^1_w)$ for all $t\in [0,T]$.
\textbf{From now on in  subsection \ref{sub41} we fix $T>0$ and  all the  constants $C$ in  subsection \ref{sub41} are allowed to depend on $T$ without  mentioning.}
Note that it suffices to show that for all $\varphi\in \mathscr{C}_b(\tilde H^1_w)$,
 $P_t\varphi\in  \mathscr{S}_b(\tilde H^1_w)$ for all $t\in [0,T]$.
 Let $u_{0n}\in \tilde H^1_w $ be any sequence which satisfies that  
 $u_{0n}\rightarrow u_0$ in $\tilde H^1_w $, hence 
 $u_{0n}$, $n\in \N$ are uniformly bounded in $H^1$.
 Then we immediately have $\varphi(u_{0n})\rightarrow \varphi(u_0)$.
 Denote by $u_n$ and $u$ the probabilistically strong solutions of  
 \eqref{eq:main} with initial value $u_{0n}$ and $u_0$, respectively.

 Let $U$
be  another separable  Hilbert space which is densely and  compactly embedded in 
$\tilde H^1$ (see, e.g.,  \cite{BRZEZNIAK20131627} Appendix Lemma C.1 for the proof of the existence of such $U$),
hence $H^{-1}$  is compactly (hence continuously)  embedded in $U'$.
 Define the function space
$$\mathcal{Z}_{T}:=C([0,T];U')\cap L^2([0,T];L^2_{loc})\cap C([0,T];\tilde H^1_{w}).$$
Note that the space $\mathcal Z_T$ is endowed with the topology given by the intersection of the three topologies.
Let $P^n$ be the probability distribution of $u_n$ in 
$\mathcal{Z}_{T}$. Then we have the following proposition.

\begin{lem}
	$P^n$ is tight in $\mathcal{Z}_{T}$.
\end{lem}
 \begin{proof}
 For all $R>0$,
 define $K_{R}=\{u\in \mathcal{Z}_{T}: \sup\limits_{0\leq t\leq T}\|u(t)\|_{H^1}
+\|u\|_ {C^{\frac{1}{4}}([0,T];U')}\leq R\}$,
where $C^{\frac{1}{4}}([0,T];U')$ is the H\"{o}lder (semi)norm
$$\|u\|_ {C^{\frac{1}{4}}([0,T];U')}:=\sup\limits_{s\neq t\in[0,T]}\frac{\|u(t)-u(s)\|_{U'}}{\mid t-s\mid^{\frac14}}.
$$
By Lemma 3.3 of \cite{BRZEZNIAK20131627}, $K_{R}$ is relatively compact in $\mathcal{Z}_{T}$. 
It remains to show that $\forall \epsilon>0$, there exists a large enough $R$, such that for all $n\in \N$,
$P^n(K_{R})>1-\epsilon$.
First by \eqref{solution-H1-estimate} and the fact that 
$u_{0n}$, $n\in \N$ are uniformly bounded in $H^1$,
we can find a large enough constant $R_0\in (1,\infty)$, such that for all $n\in \N$, 
 \begin{equation}\label{prob:R0}
 P^n(\hat{K}_{R_0})>1-\frac{\epsilon}{2},
 \end{equation}
 where $\hat{K}_{R_0}=\{u\in \mathcal{Z}_{T}: \sup\limits_{0\leq t\leq T}\|u(t)\|_{H^1}
\leq R_0\}  $.
Observe that by Theorem \ref{thm:wp},
\begin{equation}\label{un}
u_n(t)=u_{0n}+\int_{0}^{t}\mathcal{P}(-u_n\cdot\nabla u_n + \partial_1^2 u_n - \lambda u_n)ds+M(t),
\end{equation}
where $\mathcal{P}$ is the Leray projector and $M(t)=\int_{0}^{t} \mathcal{P}\sigma dW=\int_{0}^{t} \sigma dW $.\\
By Cauchy--Schwarz inequality, we have for all $u\in \mathcal{Z}_T$,
\begin{equation*}
\begin{split}
&\sup\limits_{s\neq t\in[0,T]}\biggl(\frac{\| \int_{s}^{t}\mathcal{P}\bigl(-\partial_{1}^{2}u+\dive(u\otimes u)-\lambda u\bigr)dr\|_{H^{-1}}^{2}}{\mid t-s \mid}\biggr)1_{u\in \hat{K}_{R_0}}\\
&\leq \int_{0}^{T}\|\mathcal{P}\bigl(-\partial_{1}^{2}u+\dive(u\otimes u)-\lambda u\bigr)\|_{H^{-1}}^{2}\,dr1_{u\in \hat{K}_{R_0}}\\
&\leq \int_{0}^{T}\|-\partial_{1}^{2}u+\dive(u\otimes u)-\lambda u\|_{H^{-1}}^{2}\,dr1_{u\in \hat{K}_{R_0}},
\end{split}
\end{equation*}
where the first inequality is due to
\begin{equation*}
\begin{split}
&\Big\| \int_{s}^{t}\mathcal{P}\bigl(-\partial_{1}^{2}u+\dive(u\otimes u)-\lambda u\bigr)dr\Big\|_{H^{-1}}^{2}\\
&\leq |t-s|\int_s^t\|\mathcal{P}\bigl(-\partial_{1}^{2}u+\dive(u\otimes u)-\lambda u\bigr)\|_{H^{-1}}^{2}\,dr
\end{split}
\end{equation*}
and the second inequality follows from the boundedness of the Leray projector.
Hence
\begin{equation}\label{S4eq18}
\begin{split}
&\sup\limits_{s\neq t\in[0,T]}\biggl(\frac{\| \int_{s}^{t}-\partial_{1}^{2}u+\dive(u\otimes u)-\lambda u\,dr\|_{H^{-1}}^{2}}{\mid t-s \mid}\biggr)1_{u\in \hat{K}_{R_0}}\\
&\leq 3\int_{0}^{T}\|-\partial_{1}^{2}u\|_{H^{-1}}^{2}+\|\dive(u\otimes u)\|_{H^{-1}}^{2}+\lambda^2 \|u\|_{H^{-1}}^2\,dr1_{u\in \hat{K}_{R_0}}\\
&\leq 3\int_{0}^{T}(1+\lambda^2)\|u\|_{H^{1}}^{2}+\|u\otimes u\|_{L^2}^{2}\,dr1_{u\in \hat{K}_{R_0}}\\
&\leq 3\int_{0}^{T}(1+\lambda^2)\|u\|_{H^{1}}^{2}+\| u\|_{L^4}^{4}\,dr1_{u\in \hat{K}_{R_0}}\\
&\leq C\int_{0}^{T}\|u\|_{H^{1}}^2+\|u\|_{H^{1}}^{4}\,dr1_{u\in \hat{K}_{R_0}},
\end{split}
\end{equation}
where the last inequality is due to the Sobolev embedding
$H^1$ to $L^4$.
 Therefore, for all $u\in \hat{K}_{R_0}$, we obtain  that 
 $$\sup\limits_{s\neq t\in[0,T]}\biggl(\frac{\| \int_{s}^{t}-\partial_{1}^{2}u+\dive(u\otimes u)dr-\lambda u\,dr\|_{H^{-1}}^{2}}{\mid t-s \mid}\biggr)\leq CR_0^4,
 $$
 where $C$ is a constant which does not depend on $u$ and $R_0$.
 Hence we immediately have for all $u\in \hat{K}_{R_0}$,
  \begin{equation*}
  \sup\limits_{s\neq t\in[0,T]}\biggl(\frac{\| \int_{s}^{t}-\partial_{1}^{2}u+\dive(u\otimes u)-\lambda u\,dr\|_{H^{-1}}}{| t-s|^{\frac12}}\biggr)\leq (CR_0^4)^{\frac12}.
 \end{equation*}
Therefore, there exists a constant $C$, such that 
for all $u\in \hat{K}_{R_0}$,
\begin{equation}
  \sup\limits_{s\neq t\in[0,T]}\biggl(\frac{\| \int_{s}^{t}-\partial_{1}^{2}u+\dive(u\otimes u)-\lambda u\,dr\|_{H^{-1}}}{| t-s|^{\frac14}}\biggr)\leq CR_0^2.
 \end{equation}
Since 
$H^{-1}$  is  continuously  embedded in $U'$,
we immediately know that
there exists a constant $C$, such that 
for all $u\in \hat{K}_{R_0}$,
\begin{equation}\label{drift:holder}
  \sup\limits_{s\neq t\in[0,T]}\biggl(\frac{\| \int_{s}^{t}-\partial_{1}^{2}u+\dive(u\otimes u)-\lambda u\,dr\|_{U'}}{| t-s|^{\frac14}}\biggr)\leq CR_0^2.
\end{equation}

 Next we consider the martingale term $M(t)=\int_{0}^{t}\sigma dW $.
 First observe that 
 for all $T\geq t>s\geq 0$ and $p\in \mathbb{N}$,
 \begin{equation*}
 \begin{split}
&\EE\|M(t)-M(s)\|_{H^{-1}}^{2p}\\
&\leq \EE\|M(t)-M(s)\|_{H^{1}}^{2p}\\
&\le C_p\,\EE\Big(\int_s^t \|\sigma\|_{L_2(\ell^2,H^{1})}^2\,dr\Big)^p\\
&
= C_p\,\|\sigma\|_{L_2(\ell^2,H^{1})}^{2p}\,|t-s|^{p},
\end{split}
 \end{equation*}
 where $C_p$ is a constant which only depends on $p$.
Hence by Kolmogorov's criterion (see, e.g., \cite[Chap.~I, Theorem 2.1]{RevuzYor1999}), for all $p\in \mathbb{N}$ and $\alpha \in (0,\frac{p-1}{2p})$, we deduce
 \begin{equation*}
 \mathbb{E}\Biggl[\Bigl(\sup\limits_{s\neq t\in[0,T]}\frac{\|M(t)-M(s)\|_{H^{-1}}}{\mid t-s\mid^{\alpha}}\Bigr)^{2p}\Biggr]\leq  C_{p,\alpha}\,\|\sigma\|_{L_2(\ell^2,H^{1})}^{2p}.
 \end{equation*}
 Choose $p=4$ and $\alpha=\frac14$. 
 Then we immediately obtain that there exists a constant $C$,
 such that 
 \begin{equation*}
 \mathbb{E}\Biggl[\Bigl(\sup\limits_{s\neq t\in[0,T]}\frac{\|M(t)-M(s)\|_{H^{-1}}}{\mid t-s\mid^{\frac14}}\Bigr)^{8}\Biggr]\leq  C\,\|\sigma\|_{L_2(\ell^2,H^{1})}^{8}.
 \end{equation*}
 By Chebyshev's inequality, we can find a constant $R_1$,
 such that 
$$\mathbb{P}\Bigl(\sup\limits_{s\neq t\in[0,T]}\frac{\|M(t)-M(s)\|_{H^{-1}}}{\mid t-s\mid^{\frac14}}>R_1\Bigr)<\frac{\epsilon}{2}.
$$
Since 
$H^{-1}$  is  continuously  embedded in $U'$,
we immediately know that
there exists a constant $C$, such that 
\begin{equation}\label{M:holder}
\mathbb{P}\Bigl(\|M\|_ {C^{\frac{1}{4}}([0,T];U')}>CR_1\Bigr)<\frac{\epsilon}{2}.
\end{equation}
For each $n\in\N$, define the event
$$
A_n:=\{\omega\in\Omega:\;u_n\in \hat{K}_{R_0},\|M\|_ {C^{\frac{1}{4}}([0,T];U')}\leq CR_1\},
$$
where $C$ is the constant from \eqref{M:holder}.
Combining \eqref{prob:R0} and \eqref{M:holder},
we deduce that $ \forall n\in\N$,
$$\mathbb{P}(A_n)>1-\epsilon.$$
Moreover, combining \eqref{un} and \eqref{drift:holder},
  we obtain that there exists a constant $C$ independent of $n$, $R_0$, and $R_1$, such that 
  on the event $A_n$,
\begin{equation}\label{holder:un}
\|u_n\|_{C^{\frac{1}{4}}([0,T];U')}
\leq C R_0^2 + CR_1,
\end{equation}
 Choosing 
 $$R=R_0+C R_0^2 + CR_1.$$
Then
for all $\omega\in A_n  $,
 $u_n(\omega)\in K_R$.
 Hence  for each $n\in \mathbb{N}$,
 $P^n(K_R)\geq \mathbb{P}(A_n)>1-\epsilon
 $
 \end{proof}

Then we apply the Skorokhod Theorem from  
 \cite[Theorem 2]{Jakubowski1998Short} and use the form from
 Theorem A.1 in \cite{brzezniak2013}. The verification 
 of $\mathcal{Z}_T$ satisfying  the assumptions of the theorem is the same as   in 
  \cite[Corollary 3.12]{BRZEZNIAK20131627}, hence we omit the details.
We follow the same strategy as in 
\cite[Theorem 4.11]{BrzezniakMotylOndrejat2017},
in other word, we identify the limit as a martingale solution which satisfies the similar weak formulation as in \cite{BrzezniakMotylOndrejat2017}. Let $\hat{K}$ be an auxiliary Hilbert space such that the embedding 
$\ell^2\subset  \hat{K}$ densely and continuously and the  embedding
$i\colon \ell^2\hookrightarrow \hat{K}$ is Hilbert--Schmidt.
 We identify $\ell^2$-cylindrical Wiener process 
  $W(\cdot)$ with its $Q$-Wiener realization in   $C([0,T];\hat{K})$, where $Q=ii^{\ast}$. From now on we denote this $\hat K$-valued process $i(W)$ again by $W$.
 Note that 
the product topological space 
$\mathcal{Z}_T\times C([0,T];\hat{K})$ also 
satisfies   the assumptions of \cite[Theorem 2]{Jakubowski1998Short}.
Moreover, the laws of $\{(u_n, W)\}_{n\in\N}$ are tight in $\mathcal{Z}_T\times C([0,T];\hat{K})$.
By \cite[Theorem 2]{Jakubowski1998Short}, we 
have the following proposition.

\begin{prop}\label{prop:skoro}
	There exists another  probability space
	$(\tilde{\Omega},\tilde{\mathcal{F}},\tilde{\mathbb{P}})$, a 
	 subsequence  $n_k$, 
	 and $\mathcal{Z}_T\times C([0,T];\hat{K})$-valued Borel measurable random variables  $(\tilde{u},\tilde W)$, $(\tilde{u}_k,\tilde{W}_k)$, $k\in\N$, 
	 such that
	 
\begin{enumerate}
	\item $\forall$ $k\in \N$, $(\tilde{u}_{k},\tilde{W}_k)$ has the same law as $(u_{n_k},W)$ on 
	$\mathcal{Z}_T\times C([0,T];\hat{K})$.
	\item  $(\tilde{u}_{k},\tilde{W}_k)$ converges  to   $(\tilde{u},\tilde W) $ $\tilde{\mathbb{P}}$-almost surely in $\mathcal{Z}_T\times C([0,T];\hat{K})$.
 \end{enumerate}
\end{prop}

\begin{rem}\label{rmk:1}
 By Proposition \ref{prop:skoro},
$\tilde W_k(\cdot)$ has the same law as $W(\cdot)$
on  $C([0,T];\hat{K})$. 
Hence $\tilde W_k(\cdot)$ is a $Q$-Wiener process on
$\hat{K}$ with $Q=i i^\ast$.
Consequently,  $\tilde W_k$ can be viewed as an $\ell^2$-cylindrical Wiener process.
The same argument applies to $\tilde W$.
\end{rem}
Denote by $
\tilde{\mathbb{F}}^k=\{\mathcal{F}_t^k\}_{t\geq 0}$ the filtration generated by 
$(\tilde{u}_{k},\tilde{W}_k)$ satisfying the usual condition.
Denote by $
\tilde{\mathbb{F}}=\{\mathcal{F}_t\}_{t\geq 0}$ the filtration generated by 
 $(\tilde{u},\tilde W)$ satisfying the usual condition.
  The next theorem tells us that  the limit  is a martingale solution to the equations \eqref{eq:main}.
  \begin{thm}\label{thm:martingalesol}
  	{\sl $(\tilde{u},\tilde W,\tilde{\Omega},\tilde{\mathcal{F}},\tilde{\mathbb{F}},\tilde{\mathbb{P}})$ solves the equation \eqref{eq:main} 
  	in the sense of a similar definition from Definition 3.2 of  \cite{BrzezniakMotylOndrejat2017},
  	i.e. 
  	\begin{enumerate}
  	\item $\tilde W$ is an $\ell^2$-cylindrical Wiener process with respect to 
  	the complete filtration $
\tilde{\mathbb{F}}$.
  	\item  $\tilde{u}\colon [0,T] \times\tilde{\Omega}\rightarrow \tilde H^1  $ with $\tilde{\mathbb{P}}$-$a.s.$ paths in $\mathcal{Z}_T$.
  	\item For every $t\in [0,T]$ and $l\in C_c^{\infty}(\mathbb{R}^{2})$ with  $\dive l =0$,  $\tilde{\mathbb{P}}$-$a.s.$,
$$\tilde u(0)=u_0$$
and 
$$( \tilde u(t),l) =( u_0,l) +\int_{0}^{t}\langle-\tilde u\cdot\nabla \tilde u+\partial_{1}^{2}\tilde u-\lambda \tilde u,l\rangle\,ds+\int_{0}^{t}( \sigma d\tilde W(s),l).$$
\item It holds that
\begin{equation}\label{momentestimate:tildeu}
\tilde{\mathbb E}\!\Big(\sup_{t\in[0,T]}\|\tilde u(t)\|_{H^1}^{2}\Big)<\infty.
\end{equation}
\end{enumerate}
}
\end{thm}
 \begin{proof}
  1,2 and 4 follow directly from 
  Proposition \ref{prop:skoro}, Remark \ref{rmk:1} as well as the construction of the filtration.
  For 3, note that
  since $(\tilde u_k,\tilde W_k)$ has the same law as $(u_{n_k},W)$ on
$\mathcal Z_T\times C([0,T];\hat K)$, the  formulation satisfied by
$(u_{n_k},W)$ transfers to $(\tilde u_k,\tilde W_k)$,
which implies  for each $k\in\N$ and $t\in [0,T]$, 
  $\tilde{\mathbb{P}}$-$a.s.$ 
  \begin{equation}\label{initialvalue:k}
  	\tilde u_k(0)=u_{0n_k}
  	\end{equation}
and 
$$( \tilde u_k(t),l) =( u_{0n_k},l) +\int_{0}^{t}\langle-\tilde u_k\cdot\nabla \tilde u_k+\partial_{1}^{2}\tilde u_k-\lambda \tilde u_k,l\rangle\,ds+\int_{0}^{t}( \sigma d\tilde W_k(s),l).$$
  Then we pass to the limit when $k\to\infty$.
  Since  $\tilde u_k\to  \tilde u$ in $C([0,T];\tilde H^1_{w})$, $\tilde{\mathbb{P}}$-$a.s.$,
  we obtain that $\tilde u_k(0)\rightharpoonup  \tilde u(0)$  in $\tilde H^1$ $\tilde{\mathbb{P}}$-$a.s.$, hence by \eqref{initialvalue:k},
   $u_{0n_k}\rightharpoonup \tilde u(0)$  in $\tilde H^1$ $\tilde{\mathbb{P}}$-$a.s.$.
   On the other hand, $u_{0n_k}\rightharpoonup u_0$  in $\tilde H^1$. Therefore, by the uniqueness of the weak limit  in $\tilde H^1$, $\tilde u(0)=u_0$  $\tilde{\mathbb{P}}$-$a.s.$.
   The convergence in $L^2([0,T];L^2_{loc})$ 
  guarantees the convergence of the nonlinear term 
  $\int_{0}^{t}\langle-\tilde u_k\cdot\nabla \tilde u_k,l\rangle\,ds$ for fixed test function $l$.

   For the martingale term, 
    define $M_k^l(t):=\int_0^t(\sigma\,d\tilde W_k(s),l)$ and $M^l(t):=\int_0^t(\sigma\,d\tilde W(s),l)$.
    Since $\sigma$ is deterministic and does not depend on $t$,
$\tilde W_k\to\tilde W$ in $C([0,T];\hat K)$ implies
$M_k^l\to M^l$ in $C([0,T])$, $\tilde{\PP}$-a.s.

 	 The proof of passing to the limit  is also similar to (even easier, since our noise is additive) \cite[Theorem 4.11]{BrzezniakMotylOndrejat2017}, hence we omit more details. 
 	  	       \end{proof}

Since pathwise uniqueness holds (which can be established by adapting the methods of \cite{LiangZhangZhu2021}), we know the martingale solutions are  unique in law (see, e.g.,\cite[Corollary 7.7]{BRZEZNIAK20131627}). Therefore, for each $t\in [0,T]$, $\tilde u(t)$ and $u(t)$ have the same law on $(\tilde H^1,\mathscr{B}(\tilde H^1))$. (Since $\tilde H^1$ is a separable Hilbert space, $\mathscr{B}(\tilde H^1)=\mathscr B(\tilde H^1_w)$ )
 As a result, for all $\varphi\in \mathscr{C}_b(\tilde H^1_w)$,
it holds for all $t\in [0,T]$ that
\begin{equation}\label{eq:ulimitexp}
	\tilde{\mathbb E}\bigl[\varphi\bigl(\tilde u(t)\bigr) \bigr]
=\EE\bigl[\varphi\bigl(u(t)\bigr) \bigr]=P_t\varphi(u_0).  
\end{equation}
Since 
$\tilde{u}_{k}(t)$ has the same law as $u_{n_k}(t)$ on 
$(\tilde H^1,\mathscr{B}(\tilde H^1))$,
we obtain  for all $\varphi\in \mathscr{C}_b(\tilde H^1_w)$ that
\begin{equation}\label{eq:ukexp}
\tilde{\mathbb E}\bigl[\varphi\bigl(\tilde u_k(t)\bigr) \bigr]
=\EE\bigl[\varphi\bigl(u_{n_k}(t)\bigr) \bigr]=P_t\varphi(u_{0n_k}).  
\end{equation}
By the dominated convergence theorem, the boundedness of 
$\varphi$, together with the convergence 
 $\tilde u_k\to \tilde u$ in $C([0,T];\tilde H^1_w)$, for all $\varphi\in \mathscr{C}_b(\tilde H^1_w)$, $t\in [0,T]$,
$$\lim\limits_{k\to\infty}\tilde{\mathbb E}\bigl[\varphi\bigl(\tilde u_k(t)\bigr) \bigr]
=\tilde{\mathbb E}\bigl[\varphi\bigl(\tilde u(t)\bigr) \bigr].
$$
Combining this with \eqref{eq:ulimitexp} and \eqref{eq:ukexp}, we prove for all $\varphi\in \mathscr{C}_b(\tilde H^1_w)$, $t\in [0,T]$, that
 $$\lim\limits_{k\to\infty}\EE\bigl[\varphi\bigl(u_{n_k}(t)\bigr) \bigr]
=\EE\bigl[\varphi\bigl(u(t)\bigr) \bigr],
$$
which means 
$$\lim\limits_{k\to\infty}P_t\varphi(u_{0n_k})= P_t \varphi(u_0).
$$
Using the sub-subsequence argument, we deduce that 
$$\lim\limits_{n\to\infty}P_t\varphi(u_{0n})= P_t \varphi(u_0),
$$
 which finishes the proof. 
\subsection{The proof of \eqref{eq:MS}}
Now we prove \eqref{eq:MS} for $\mathcal{H}=\tilde H^1$.
For simplicity we choose $\nu=\delta_0$.
By Chebyshev's inequality, for all $R>0$,
$$
P_t^\ast\delta_0\big(\{\|x\|_{H^1}>R\}\big)
=\mathbb P\big(\|u^0(t)\|_{ H^1}>R\big)
\le \frac{1}{R^2}\,\mathbb E\|u^0(t)\|_{H^1}^2,
$$
where $u^0$ is the probabilistically strong solution of 
\eqref{eq:main} with initial value $0$.
Taking the time average, 
applying  \eqref{Ito:L2} and \eqref{Ito:H1} (which are shown later in the proof of Lemmas \ref{lem42} and \ref{lem43}  and their proofs rely only on 
It\^o's formula) with $z_0=0$, taking the expectation, and applying Gronwall's inequality,
we deduce that 
$$
\sup_{T>0}\frac1T\int_0^T P_t^\ast\delta_0\big(\{\|x\|_{ H^1}>R\}\big)\,dt
\le \frac{1}{R^2}\sup_{T>0}\frac1T\int_0^T\,\mathbb E\|u^0(t)\|_{H^1}^2\,dt
\le \frac{C}{R^2},
$$
where $C$ is a constant depending on $\lambda$ and the Hilbert-Schmidt norms of $\sigma$, but not on $T$.
Choosing a large enough $R$, we finish the proof.

\section{Uniqueness of Invariant Measures}\label{sec4}
In this section we present and prove the main result of this paper. 

\subsection{Main theorem}

The following theorem is the main theorem of this paper.
\begin{thm}\label{thm:unique}
There exists a constant $C>0$ (depending only on the two--dimensional
Sobolev--Gagliardo--Nirenberg and Burkholder--Davis--Gundy constants) such that 
for all
$$
\lambda > C(\,\|\sigma\|_{L_2(\ell^2,H^1)}^{2} + 1),
$$
the Markov semigroup associated with \eqref{eq:main} admits at most one invariant
probability measure on $(\tilde H^1,\mathcal B(\tilde H^1))$.
\end{thm}

The proof is based on a coupling method.

\subsection{Auxiliary measurable lemmas}
Before proceeding with the probabilistic energy estimates and coupling method, 
we first introduce two auxiliary results concerning 
the measurable and topological structure of $\tilde H^1$,
both of which will be applied in the coupling argument.

\begin{lem}[Equality of Borel $\sigma$-algebras on $\tilde H^1$]\label{lem:Borel-equal}
Let $\mathcal B(\tilde H^1)$ and $\mathcal B_{L^2}(\tilde H^1)$ denote the Borel $\sigma$-algebras on $\tilde H^1$ induced by $\|\cdot\|_{H^1}$ and $\|\cdot\|_{L^2}$, respectively. Then
$$
\mathcal B(\tilde H^1)=\mathcal B_{L^2}(\tilde H^1).
$$
\end{lem}
\begin{proof}
First note that the identity map 
$\mathrm{id}\colon (\tilde H^1,\|\cdot\|_{H^1}) \to (\tilde H^1,\|\cdot\|_{L^2})$ 
is continuous. Therefore every $L^2$-open subset of $\tilde H^1$ is $H^1$-open, thus
$$
\mathcal B_{L^2}(\tilde H^1)\subset \mathcal B(\tilde H^1).
$$
For the converse inclusion, note that $(\tilde H^1,\|\cdot\|_{H^1})$ is a separable Hilbert space and hence a Polish space. 
By a standard result on Borel mappings (see Bogachev \cite[Vol.~II, Thm.~6.8.6]{Bogachev2007-II}), 
the identity map 
$$
\mathrm{id}\colon (\tilde H^1,\|\cdot\|_{H^1}) \to (\tilde H^1,\|\cdot\|_{L^2})
$$
is a Borel isomorphism onto its image, and this image (which is $\tilde H^1$ itself) is a Borel subset of $(\tilde H^1,\|\cdot\|_{L^2})$.
As a result, every $A\in\mathcal B(\tilde H^1)$ is Borel in $(\tilde H^1,\|\cdot\|_{L^2})$.
Hence
$$
\mathcal B(\tilde H^1)\subset \mathcal B_{L^2}(\tilde H^1),
$$
and the two $\sigma$-algebras coincide.
\end{proof}

From now on we 
set $S:=(\tilde H^1,\|\cdot\|_{L^2})$, hence by Lemma~\ref{lem:Borel-equal}, its Borel $\sigma$-algebra is 
$\mathcal B(\tilde H^1)$.  

\begin{prop}[Bounded Lipschitz class is measure determining]\label{prop:BL}

Define
$$
\mathrm{BL}_1(S):=\{f:S\to\mathbb{R}\,:\ \|f\|_\infty\le 1,\ \mathrm{Lip}(f)\le 1\}.
$$
Then for all $\mu,\nu\in\mathcal P(S)$ which satisfy that 
\begin{equation}
\sup_{f\in \mathrm{BL}_1(S)}\Big|\int f\,\mathrm d\mu-\int f\,\mathrm d\nu\Big|=0,
\end{equation}
we have $\mu=\nu $.
\end{prop}

\begin{proof}
This is classical for separable metric spaces, see,  e.g., R.~M. Dudley,
\cite[Prop.~11.3.2]{Dudley2002}.
\end{proof}

\subsection{Energy estimates}

 From now on we set $K:=\|\sigma\|_{L_2(\ell^2,H^1)}^2$.
For a solution $z$ of \eqref{eq:main} with initial data $z_0$, set
$$
\Ezero{z}(t):=\|z(t)\|_{L^2}^2-\|z_0\|_{L^2}^2
+2\!\int_0^t\!\|\partial_1 z(s)\|_{L^2}^2\,\mathrm ds
+(2\lambda-1)\!\int_0^t\!\|z(s)\|_{L^2}^2\,\mathrm ds
- K t ,
$$
and
$$
\Eone{z}(t):=\|\nabla z(t)\|_{L^2}^2-\|\nabla z_0\|_{L^2}^2
+2\!\int_0^t\!\|\partial_1\nabla z(s)\|_{L^2}^2\,\mathrm ds
+(2\lambda-1)\!\int_0^t\!\|\nabla z(s)\|_{L^2}^2\,\mathrm ds
- K t .
$$
 Then we have the following two lemmas.
\begin{lem}\label{lem42}
Fix $z_0\in\tilde H^1$ and let $z(t;z_0)$ be the corresponding strong solution of \eqref{eq:main} in the sense of Theorem~\ref{thm:wp}. Then for all $R>0$,
\begin{equation}
\PP\!\left(
\sup_{t\ge 0} \mathcal{E}_0^z(t) \ge 2R
\right) \le
\begin{cases}
\exp\!\bigl(-\dfrac{R}{2K}\bigr), & \text{if } K>0,\\[4pt]
0, & \text{if } K=0.
\end{cases}
\end{equation}
\end{lem}

\begin{proof}
Apply It\^o's formula in Hilbert space (see, e.g.,\cite[Chap.~IV]{DPZ1992}) to $\norm{z(t)}_{L^2}^2$, noting that $(z\cdot\nabla z, z)=0$, we obtain that
\begin{equation}\label{Ito:L2}
\begin{split}
\norm{z(t)}_{L^2}^2 + 2\int_0^t \norm{\partial_1 z(s)}_{L^2}^2\,\mathrm{d}s + 2\lambda \int_0^t \norm{z(s)}_{L^2}^2\,\mathrm{d}s\\
= \norm{z_0}_{L^2}^2 + 2\int_0^t (\sigma\,\mathrm{d}W(s),z(s)) + \int_0^t \norm{\sigma}_{L_2(\ell^2,L^2)}^2\,\mathrm{d}s.
\end{split}
\end{equation}
If $K=0$, then $\sigma\equiv0$ and
\eqref{eq:main} is a deterministic equation.
We immediately know that for all $t\geq 0$ that 
$$
\|z(t)\|_{L^2}^2 + 2\int_0^t \|\partial_1 z(s)\|_{L^2}^2\,\mathrm{d}s
+ 2\lambda \int_0^t \|z(s)\|_{L^2}^2\,\mathrm{d}s
= \|z_0\|_{L^2}^2 .
$$
Since
$$\Ezero{z}(t)=\|z(t)\|_{L^2}^2-\|z_0\|_{L^2}^2
+2\!\int_0^t\!\|\partial_1 z(s)\|_{L^2}^2\,\mathrm ds
+(2\lambda-1)\!\int_0^t\!\|z(s)\|_{L^2}^2\,\mathrm ds,
$$
we obtain that $\Ezero{z}(t)=-\int_0^t\!\|z(s)\|_{L^2}^2\,\mathrm ds\leq 0$ for all $t\geq 0$.
In this case the claimed probability bound is trivial. 
Therefore, in the following we may assume  $K>0$.
Define
\begin{equation}
M_t := \int_0^t (\sigma\,\mathrm{d}W(s),z(s)).
\end{equation}
Then we have the exponential martingale bound
\begin{equation}\label{exponentialmartingale1}
\PP\!\left(\sup_{t\ge 0} \big(M_t - \gamma \langle M\rangle_t\big) \ge R \right) \le e^{-\gamma R},
\end{equation}
which follows from the exponential martingale inequality for continuous local martingales (see, e.g., Revuz and Yor \cite[Chap.~II, Prop.~1.8 \& Chap.~IV, Ex.~3.16 for extension to continuous local martingales]{RevuzYor1999}).\\
Observe that
\begin{equation}
\langle M\rangle_t = \int_0^t \norm{\sigma^\ast z}_{\ell^2}^2\,\mathrm{d}s
\le \int_0^t \norm{\sigma}_{L_2(\ell^2,L^2)}^2 \norm{z(s)}_{L^2}^2\,\mathrm{d}s
\le K \int_0^t \norm{z(s)}_{L^2}^2\,\mathrm{d}s,
\end{equation}
since $\|\sigma\|_{L_2(\ell^2,L^2)}^2\le \|\sigma\|_{L_2(\ell^2,H^1)}^2=K$. 
Thus
$$
\mathcal{E}_0^z(t)
\le 2M_t - \int_0^t \|z(s)\|_{L^2}^2\,\mathrm{d}s
\le 2M_t - \frac{1}{K}\,\langle M\rangle_t.
$$
Hence, by  \eqref{exponentialmartingale1}, for all $R>0$,
$$
\PP\Big(\sup_{t\ge0}\mathcal{E}_0^z(t)\ge 2R\Big)
\le
\PP\Big(\sup_{t\ge0}\big(M_t-\tfrac{1}{2K}\langle M\rangle_t\big)\ge R\Big)
\le \exp\!\Big(-\frac{R}{2K}\Big).
$$
 This completes the proof of the lemma.
\end{proof}

\begin{lem}\label{lem43}
Fix $z_0\in\tilde H^1$ and let $z(t;z_0)$ be the corresponding strong solution of \eqref{eq:main}.
Then for all $R>0$,
\begin{equation}
\PP\!\left(
\sup_{t\ge 0} \mathcal{E}_1^z(t) \ge 2R
\right) \le
\begin{cases}
\exp\!\bigl(-\dfrac{R}{2K}\bigr), & \text{if } K>0,\\[4pt]
0, & \text{if } K=0.
\end{cases}
\end{equation}
\end{lem}
\begin{proof}

Consider the functional $F:H^1\to\R$, $F(u)=\|\nabla u\|_{L^2}^2$.
Using It\^o's formula in Hilbert spaces for $F$ (see, e.g., \cite[Chap.~IV]{DPZ1992}),
together with the cancellation Lemma~\ref{lem:cancel}, we obtain that
\begin{equation}
\begin{split}
\label{Ito:H1}
\|\nabla z(t)\|_{L^2}^2
&+ 2\!\int_0^t\!\|\partial_1\nabla z(s)\|_{L^2}^2\,\mathrm ds
+ 2\lambda\!\int_0^t\!\|\nabla z(s)\|_{L^2}^2\,\mathrm ds\\
&\leq   \|\nabla z_0\|_{L^2}^2 + 2 M_t + Kt,
\end{split}
\end{equation}
where $M_t:=\int_0^t \langle \nabla z(s),\nabla\sigma\,\mathrm dW(s)\rangle$.
If $K=0$, then $\sigma\equiv 0$ and \eqref{eq:main} is deterministic.
Similar to Lemma~\ref{lem42}, we obtain
that $\mathcal E_1^z(t)\leq - \int_0^t\!\|\nabla z(s)\|_{L^2}^2\,\mathrm ds\le0$ for all $t\ge0$ and the claim is trivial.
Thus we may assume $K>0$ in the following part of the proof.
From \eqref{Ito:H1} we get
$$
\mathcal{E}_1^z(t)
\leq  2 M_t - \int_0^t\!\|\nabla z(s)\|_{L^2}^2\,\mathrm ds ,$$
and
$$
\langle M\rangle_t
= \int_0^t \|(\nabla\sigma)^{\ast}\nabla z(s)\|_{\ell^2}^2\,\mathrm ds
\le K \int_0^t\!\|\nabla z(s)\|_{L^2}^2\,\mathrm ds .
$$
Hence
$$
\mathcal{E}_1^z(t)\le 2 M_t - \frac{1}{K}\,\langle M\rangle_t .
$$
Similar to the proof of Lemma~\ref{lem42},
by the exponential martingale inequality for continuous local martingales
(\cite[Chap.~II, Prop.~1.8; Chap.~IV, Ex.~3.16]{RevuzYor1999}),
for all $R>0$ we have
$$
\PP\Big(\sup_{t\ge0}\mathcal{E}_1^z(t)\ge 2R\Big)
\le
\PP\Big(\sup_{t\ge0}\big(M_t-\tfrac{1}{2K}\langle M\rangle_t\big)\ge R\Big)
\le \exp\!\Big(-\frac{R}{2K}\Big).
$$
This completes the proof.
\end{proof}

 \subsection{Asymptotic convergence and coupling argument}
 
Let $u$ and $v$ be two solutions of \eqref{eq:main}
defined on the same probability space, driven by the same Wiener process, 
 with initial data
$u_0,v_0\in\tilde H^1$, respectively. Let $T\in (0,\infty)$ be fixed and
 set $t_n=nT$.
Denote by $\Gamma$ the joint law of the sequence
$\big(u(t_n),v(t_n)\big)_{n\in\mathbb N}$.

\noindent
Take $\delta = u-v$ and $q = p_u-p_v$. Then $\delta$ satisfies that
\begin{equation}
\left\{
\begin{aligned}
&\partial_t \delta - \partial_1^2 \delta + \lambda \delta = -\nabla q - v\cdot\nabla \delta - \delta\cdot\nabla u,\\
&\dive \delta = 0,\\
&\delta\big|_{t=0} = u_0 - v_0.
\end{aligned}
\right.
\end{equation}

\begin{lem}\label{lem:delta-estimate}
We have the following estimate for $\delta$:
\begin{equation*}
\begin{aligned}
\|\delta(t)\|_{L^{2}}^{2}
&\le \|\delta_{0}\|_{L^{2}}^{2}\,
   \exp\Bigg\{-2\lambda t \\
&\qquad\qquad
+ C_{0}\!\int_{0}^{t}\!\Big(
\|\partial_{1}u(s)\|_{L^{2}}^{\frac23}
  \|\partial_{1}\partial_{2}u(s)\|_{L^{2}}^{\frac23}
+
\|\partial_{2}u(s)\|_{L^{2}}^{\frac23}
  \|\partial_{1}\partial_{2}u(s)\|_{L^{2}}^{\frac23}
\Big)\mathrm{d}s
\Bigg\}.
\end{aligned}
\end{equation*}
where $C_{0}>0$ is a constant that does not depend on $\lambda$, $u$ and $\delta$.
\end{lem}

\begin{proof}
Recall that $\delta=u-v$ and $q=p_{u}-p_{v}$ satisfy that
$$
\left\{
\begin{aligned}
&\partial_{t}\delta-\partial_{1}^{2}\delta+\lambda\delta
=-\nabla q - v\cdot\nabla\delta-\delta\cdot\nabla u,\\
&\operatorname{div}\delta=0,\\
&\delta\big|_{t=0}=\delta_{0}.
\end{aligned}
\right.
$$
Taking the $L^{2}$ inner product of this equation, 
we obtain that 
\begin{equation}\label{energyinequality1}
\frac{\mathrm{d}}{\mathrm{d}t}\|\delta\|_{L^{2}}^{2}
+2\|\partial_{1}\delta\|_{L^{2}}^{2}
+2\lambda\|\delta\|_{L^{2}}^{2}
\le 2\bigl|(\delta\cdot\nabla u,\delta)\bigr|.
\end{equation}
Observing that
\begin{equation*}
\begin{split}
|(\delta\cdot\nabla u,\delta)|
&=|(\delta^1\partial_1 u+\delta^2\partial_2 u,\delta)|\\
&\le 
\bigl(\|\delta^1\|_{L_{\mathrm h}^{\infty}(L_{\mathrm v}^{2})}\|\partial_{1}u\|_{L_{\mathrm h}^{2}(L_{\mathrm v}^{\infty})}
  + \|\delta^2\|_{L_{\mathrm h}^{2}(L_{\mathrm v}^{\infty})}\|\partial_{2}u\|_{L_{\mathrm h}^{\infty}(L_{\mathrm v}^{2})}\bigr)
  \|\delta\|_{L^{2}}.
\end{split}
\end{equation*}

Now by Lemma~3.3 of \cite{LiangZhangZhu2021}
(which is an anisotropic
Gagliardo--Nirenberg inequality adapted to the norms
$L_{\mathrm h}^\infty(L_{\mathrm v}^2)$ and
$L_{\mathrm h}^2(L_{\mathrm v}^\infty)$),
using the fact that $\dive \delta=0 $,
we deduce that
\begin{equation*}
\begin{aligned}
\bigl|(\delta\!\cdot\!\nabla u,\delta) \bigr|
&\le
\|\delta\|_{L^{2}}^{\frac12}\|\partial_{1}\delta\|_{L^{2}}^{\frac12}
 \|\partial_{1}u\|_{L^{2}}^{\frac12}
 \|\partial_{1}\partial_{2}u\|_{L^{2}}^{\frac12}\,
 \|\delta\|_{L^{2}}  \\
&\quad+
\|\delta\|_{L^{2}}^{\frac12}\|\partial_{1}\delta\|_{L^{2}}^{\frac12}
 \|\partial_{2}u\|_{L^{2}}^{\frac12}
 \|\partial_{1}\partial_{2}u\|_{L^{2}}^{\frac12}\,
 \|\delta\|_{L^{2}} .
\end{aligned}
\end{equation*}
Using Young's inequality, we obtain that
\begin{equation*}
2\bigl|(\delta\cdot\nabla u,\delta)_{L^{2}}\bigr|
\le
\|\partial_{1}\delta\|_{L^{2}}^{2}
+C_{0}\Big(
\|\partial_{1}u\|_{L^{2}}^{\frac23}\|\partial_{1}\partial_{2}u\|_{L^{2}}^{\frac23}
+\|\partial_{2}u\|_{L^{2}}^{\frac23}\|\partial_{1}\partial_{2}u\|_{L^{2}}^{\frac23}
\Big)\|\delta\|_{L^{2}}^{2},
\end{equation*}
for some constant $C_{0}>0$ independent of $\lambda$,$u$, and $\delta$.
Combining this with \eqref{energyinequality1}, we deduce that
\begin{equation*}
\frac{\mathrm{d}}{\mathrm{d}t}\|\delta\|_{L^{2}}^{2}
\le
\Big(
-2\lambda
+C_{0}\big(
\|\partial_{1}u\|_{L^{2}}^{\frac23}\|\partial_{1}\partial_{2}u\|_{L^{2}}^{\frac23}
+\|\partial_{2}u\|_{L^{2}}^{\frac23}\|\partial_{1}\partial_{2}u\|_{L^{2}}^{\frac23}
\big)
\Big)\|\delta\|_{L^{2}}^{2}.
\end{equation*}
Gronwall's lemma on $[0,t]$ then yields
\begin{equation*}
\begin{aligned}
\|\delta(t)\|_{L^{2}}^{2}
&\le \|\delta_{0}\|_{L^{2}}^{2}
   \exp\Bigg\{-2\lambda t  \\
&\qquad\quad
   +C_{0}\int_{0}^{t} \Big(
     \|\partial_{1}u(s)\|_{L^{2}}^{\frac23}
     \|\partial_{1}\partial_{2}u(s)\|_{L^{2}}^{\frac23}
    +\|\partial_{2}u(s)\|_{L^{2}}^{\frac23}
     \|\partial_{1}\partial_{2}u(s)\|_{L^{2}}^{\frac23}
   \Big)\,\mathrm{d}s\Bigg\},
\end{aligned}
\end{equation*}
which finishes the proof.
\end{proof}

Now we define the event

\begin{equation}\label{def:ER}
E_R^{(u)} := \Big\{\, \omega\in\Omega:\;
\sup_{t\ge0}\mathcal{E}_0^u(t,\omega)\le 2R,\;
\sup_{t\ge0}\mathcal{E}_1^u(t,\omega)\le 2R
\Big\}.
\end{equation}
By Lemmas~\ref{lem42} and \ref{lem43} we know that
\begin{equation*}
\PP\big(E_R^{(u)}\big)\ \ge\
\begin{cases}
1-2e^{-R/(2K)}, & \text{if } K>0,\\[4pt]
1, & \text{if } K=0.
\end{cases}
\end{equation*}
Choose and fix $R=R(K)$ large enough so that  $\PP\big(E_R^{(u)}\big)>\tfrac13$.
To show the final result, we need
the following proposition, which is an asymptotic coupling criterion.
\begin{prop}[Asymptotic coupling criterion]
\label{prop:coupling}
Let
$\widetilde{\mathcal C}\big(\delta_{u_0}\mathbb P^{\mathbb N},
\delta_{v_0}\mathbb P^{\mathbb N}\big)$ be
the set of couplings on
$(\tilde H^1)^{\mathbb N}\!\times (\tilde H^1)^{\mathbb N}$
whose marginals correspond to the Markov chains
$(u(t_n))_{n\in\mathbb N}$ and $(v(t_n))_{n\in\mathbb N}$ started from $u_0$
and $v_0$, respectively.
If there exists $\Gamma\in\widetilde{\mathcal C}
(\delta_{u_0}\mathbb P^{\mathbb N},\delta_{v_0}\mathbb P^{\mathbb N})$
such that
$$
\Gamma(D)>0,\qquad
D:=\big\{(u,v)\in (\tilde H^1)^{\mathbb N}\!\times (\tilde H^1)^{\mathbb N}:
\ \lim_{n\to\infty}\|u(t_n)-v(t_n)\|_{L^2}=0\big\},
$$
then the Markov semigroup associated with \eqref{eq:main}
admits at most one invariant probability measure on
$(\tilde H^1,\mathcal B(\tilde H^1))$.
\end{prop}
\begin{proof}
 By Proposition~\ref{prop:BL}, the class $\mathrm{BL}_1(S)$
is measure-determining on $(S,\mathcal B(S))$.
By Lemma~\ref{lem:Borel-equal}, this is the same as
$(\tilde H^1,\mathcal B(\tilde H^1))$.
If $\Gamma(D)>0$ with
$$
D=\Big\{(u,v):\ \|u(t_n)-v(t_n)\|_{L^2}\to0\Big\},
$$
then for every $\phi\in\mathrm{BL}_1(S)$,
$$
\big|\phi(u(t_k))-\phi(v(t_k))\big|
\le \|u(t_k)-v(t_k)\|_{L^2},$$
which means
$$
\frac1n\sum_{k=1}^n\!\big(\phi(u(t_k))-\phi(v(t_k))\big)\longrightarrow 0
$$
on a set of positive $\Gamma$-measure. (In the notation of
\cite[Eq.~(2.1)]{GlattHoltzMattinglyRichards2017}, this is $\Gamma(\bar D_\phi)>0$
for all $\phi\in\mathcal G$.) Hence the hypotheses of
\cite[Theorem~2.1]{GlattHoltzMattinglyRichards2017} hold (with state space $H_0=\tilde H^1$),
and the associated Markov semigroup has at most one invariant probability measure on
$(\tilde H^1,\mathcal B(\tilde H^1))$.
\end{proof}

Finally we prove the main result.
\begin{proof}[Proof of Theorem~\ref{thm:unique}]
\

\noindent
By H\"older's and Young's inequalities, we have
\begin{equation}\label{ineq:Holder-Young}
\begin{split}
&\int_0^t C_0\!\Big(
\|\partial_1 u\|_{L^2}^{\frac{2}{3}}\|\partial_1\partial_2 u\|_{L^2}^{\frac{2}{3}}
+ \|\partial_2 u\|_{L^2}^{\frac{2}{3}}\|\partial_1\partial_2 u\|_{L^2}^{\frac{2}{3}}
\Big)\mathrm{d}s\\
&\le t
+ C_1\!\int_0^t \|\partial_1 u\|_{L^2}^2\,\mathrm{d}s
+ C_1\!\int_0^t \|\partial_1\partial_2 u\|_{L^2}^2\,\mathrm{d}s \\
&\quad + \tfrac{1}{2}\sup_{0\le s\le t}\|\partial_2 u(s)\|_{L^2}^2,
\end{split}
\end{equation}
where $C_1$ is a constant depending only on $C_0$.
Thus, by Lemma~\ref{lem:delta-estimate},
\begin{equation}\label{ineq:delta-t}
\begin{aligned}
\|\delta(t)\|_{L^2}^2
&\le \|\delta_0\|_{L^2}^2
   \exp\Bigg\{-2\lambda t + t
   + C_1\!\int_0^t \|\partial_1 u(s)\|_{L^2}^2\,\mathrm{d}s \\
&\qquad\qquad\quad
   + C_1\!\int_0^t \|\partial_1\partial_2 u(s)\|_{L^2}^2\,\mathrm{d}s
   + \tfrac{1}{2}\sup_{0\le s\le t}\|\partial_2 u(s)\|_{L^2}^2
   \Bigg\}.
\end{aligned}
\end{equation}
On the event $E_R^{(u)}$ defined in~\eqref{def:ER}, Lemmas~\ref{lem42}--\ref{lem43} give
\begin{equation}\label{ineq:ER-u}
\begin{split}
&C_1\!\int_0^t\!\|\partial_1 u(s)\|_{L^2}^2\,\mathrm{d}s
+ C_1\!\int_0^t\!\|\partial_1\partial_2 u(s)\|_{L^2}^2\,\mathrm{d}s
+ C_1\,\sup_{0\le s\le t}\|\partial_2 u(s)\|_{L^2}^2\\
&\le C_2\big(R + Kt + \|u_0\|_{H^1}^2\big),
\end{split}
\end{equation}
where $C_2$ is a constant depending only on $C_1$.
Combining  this with \eqref{ineq:delta-t} yields
\begin{equation}\label{ineq:delta-2-u}
\|\delta(t)\|_{L^2}^2
\le \|\delta_0\|_{L^2}^2
\exp\!\Big\{
(-2\lambda + 1 + C_2K)t + C_2(R + \|u_0\|_{H^1}^2)
\Big\}.
\end{equation}
Choosing
$$
\lambda > \tfrac{1 + C_2K}{2}
$$
ensures that the coefficient of $t$ in the exponent is negative. Consequently,
$$
\|\delta(t)\|_{L^2} \to 0
\quad\text{as } t\to\infty \text{ on } E_R^{(u)}.
$$
Define
$$
D := \Big\{
(u,v)\in(\tilde H^1)^{\mathbb N}\!\times(\tilde H^1)^{\mathbb N} :
\ \lim_{n\to\infty} \|u(t_n)-v(t_n)\|_{L^2} = 0
\Big\}.
$$
Then $E_R^{(u)} \subset D$, and since $\Gamma$ is the joint law of
$\big(u(t_n),v(t_n)\big)_{n\in\mathbb N}$,
$$
\Gamma(D)
= \PP\!\big((u(t_n),v(t_n))\in D\big)
\ge \PP\big(E_R^{(u)}\big)
> \tfrac{1}{3}.
$$
By Proposition~\ref{prop:coupling}, the Markov semigroup associated with~\eqref{eq:main}
admits at most one invariant probability measure on
$(\tilde H^1,\mathcal B(\tilde H^1))$.
\qedhere
\end{proof}

\section*{Acknowledgments}
S. Liang is partially  funded by Basic Research Program of Jiangsu, No. BK20251436 and  by the Fundamental Research Funds for the Central Universities in China, No. 30924010943. The author  was also supported by Deutsche Forschungsgemeinschaft (DFG) through the
program IRTG 2235.
The author is grateful to Rongchan Zhu for valuable discussions and suggestions.

% ------------------------------------------------------------------------

\begin{thebibliography}{99}
\bibitem{BarbuDaPrato2002}
V.~Barbu and G.~Da Prato,
The stochastic nonlinear damped wave equation,
\emph{Appl. Math. Optim.} \textbf{46} (2002), no.~2--3, 125--141.



\bibitem{BessaihFerrario2014}
H.~Bessaih and B.~Ferrario,
Inviscid limit of stochastic damped 2D Navier--Stokes equations,
\emph{Nonlinearity} \textbf{27} (2014), no.~1, 1--15.


\bibitem{BessaihFerrario2020}
H.~Bessaih and B.~Ferrario,
Invariant measures for stochastic damped 2D Euler equations,
\emph{Comm. Math. Phys.} \textbf{377} (2020), no.~1, 531--549.




\bibitem{Bogachev2007-II}
V.~I. Bogachev,
\newblock \emph{Measure Theory. Vol.~II},
\newblock Springer, Berlin, 2007.



\bibitem{BrzezniakFerrario2019}
Z.~Brze\'{z}niak and B.~Ferrario,
Stationary solutions for stochastic damped Navier--Stokes equations in $\mathbb{R}^d$,
\emph{Indiana Univ. Math. J.} \textbf{68} (2019), no.~1, 105--138.



\bibitem{BrzezniakFerrarioZanella2023}
Z.~Brze\'{z}niak, B.~Ferrario and M.~Zanella,
Ergodic results for the stochastic nonlinear {S}chr\"{o}dinger    equation with large damping,
\emph{J. Evol. Equ.} \textbf{23(1)} (2023),   Paper No. 19, 31 pp.













\bibitem{BRZEZNIAK20131627}
Z. Brze\'{z}niak and E. Motyl.
\newblock Existence of a martingale solution of the stochastic Navier--Stokes
  equations in unbounded 2D and 3D domains.
\newblock {\em Journal of Differential Equations},  {\bf 254(4)} 1627 -- 1685, 2013.






\bibitem{BrzezniakMotylOndrejat2017}
Z.~Brze\'{z}niak, E.~Motyl and M.~Ondrej{\'{a}}t,
\newblock Invariant measure for the stochastic {N}avier--{S}tokes equations
in unbounded {2D} domains,
\newblock \emph{Ann. Probab.} \textbf{45} (2017), no.~5, 3145--3201.



\bibitem{brzezniak2013}
Z. Brze{\'{z}}niak and M. Ondrej{\'{a}}t,
\newblock Stochastic geometric wave equations with values in compact Riemannian
  homogeneous spaces.
\newblock {\em Ann. Probab.}, 41(3B):1938--1977, 05 2013.

		

\bibitem{DPZ1992}
G.~Da Prato and J.~Zabczyk,
\newblock \emph{Stochastic Equations in Infinite Dimensions},
\newblock Cambridge University Press, 1992.

\bibitem{Dudley2002}
R.~M. Dudley,
\newblock \emph{Real Analysis and Probability},
\newblock Cambridge University Press, 2002.


\bibitem{FlandoliGatarek1995}
F.~Flandoli and D.~Gatarek,
\newblock Martingale and stationary solutions for stochastic Navier--Stokes equations,
\newblock \emph{Probability Theory and Related Fields} \textbf{102}, 367--391, 1995.

\bibitem{GlattHoltzMattinglyRichards2017}
N.~Glatt-Holtz, J.~C. Mattingly, and G.~Richards,
\newblock On unique ergodicity in nonlinear stochastic partial differential equations,
\newblock \emph{Journal of Statistical Physics} \textbf{166}(3--4), 618--649, 2017.






\bibitem{Jakubowski1998Short}
A. Jakubowski, {Short Communication:The Almost Sure Skorokhod Representation for Subsequences in Nonmetric Spaces},
 {\it Theory of Probability \& Its Applications},
  {\bf 42} 
  {209-216}, 1998.








\bibitem{LiangZhangZhu2021}
S.~Liang, P.~Zhang, and R.~Zhu,
\newblock Deterministic and stochastic 2D Navier--Stokes equations with anisotropic viscosity,
\newblock \emph{Journal of Differential Equations} \textbf{275}, 473--508, 2021.


\bibitem{MaslowskiSeidler1999}
B.~Maslowski and J.~Seidler,
On sequentially weakly Feller solutions to SPDE's,
\emph{Atti Accad. Naz. Lincei Cl. Sci. Fis. Mat. Natur. Rend. Lincei (9) Mat. Appl.}
\textbf{10} (1999), no.~2, 69--78.



\bibitem{RevuzYor1999}
D.~Revuz and M.~Yor,
\newblock \emph{Continuous Martingales and Brownian Motion},
\newblock 3rd ed., Grundlehren der mathematischen Wissenschaften, Vol.~293,
Springer-Verlag, Berlin, 1999.


\bibitem{SunQiuTang2024}
C.~Sun, Z.~Qiu, and Y.~Tang,
\newblock Ergodicity for two class stochastic partial differential equations with anisotropic viscosity,
\newblock \emph{Statistics \& Probability Letters} \textbf{207}, Paper No.~110022, 8 pp., 2024.
\end{thebibliography}
\end{document}